\documentclass[10pt,a4paper]{article}

\usepackage{a4wide}
\usepackage{amsthm}
\usepackage{amsfonts}
\usepackage{amsmath}
\usepackage[english]{babel}

\usepackage{faktor}
\usepackage{etex}
\usepackage{hyperref}
\usepackage{multimedia}
\usepackage[all]{xy}
\usepackage{amssymb}
\usepackage{graphicx,textpos}
\usepackage[dvipsnames]{xcolor}
\usepackage{helvet}
\usepackage{stmaryrd}
\usepackage{enumerate}
\usepackage{todonotes}
\usepackage{pdfsync}
\usepackage{dsfont}
\usepackage[shortcuts]{extdash}
\usepackage[all]{xy}
  \newdir{ >}{{}*!/-9pt/@{>}}

\newcommand{\N}{\mathbb{N}}
\newcommand{\R}{\mathbb{R}}
\newcommand{\Q}{\mathbb{Q}}
\newcommand{\Z}{\mathbb{Z}}
\newcommand{\C}{\mathbb{C}}

\newcommand{\rst}[1]{\ensuremath{{\mathbin\mid}\raise-.5ex\hbox{$#1$}}}
\newcommand{\lie}{\mathfrak{g}}
\newcommand{\lien}{\mathfrak{n}}

\newcommand{\TG}{$\mathcal{T}$-group}

\DeclareMathOperator{\GL}{GL}

\DeclareMathOperator{\Aut}{Aut}

\DeclareMathOperator{\Id}{Id}

\author{Jonas Der\'e\thanks{KU Leuven Campus Kulak Kortrijk, Department of Mathematics, Research unit ‘Algebraic Topology and Group Theory’, B-8560 Kortrijk, Belgium. Email: \href{mailto:jonas.dere@kuleuven.be}{jonas.dere@kuleuven.be} The author was supported by a start-up grant of KU Leuven (grant number 3E210985).}}
\title{A note on the existence of the \\Reidemeister zeta function on groups.}
\date{\vspace{-5mm}
}

\newtheorem{Def}{Definition}[section]
\newtheorem{Ex}[Def]{Example}
\newtheorem{corollary}[Def]{Corollary}
\newtheorem{theorem}[Def]{Theorem}
\newtheorem{Prop}[Def]{Proposition}
\newtheorem{Lem}[Def]{Lemma}
\newtheorem{Rmk}[Def]{Remark}
\newtheorem*{Prop*}{Proposition}
\newtheorem*{Lem*}{Lemma}

\newtheorem{QN}{Question}
\newtheorem{Con}{Conjecture}
\makeatletter
\newtheorem*{rep@theorem}{\rep@title}
\newcommand{\newreptheorem}[2]{%
	\newenvironment{rep#1}[1]{%
		\def\rep@title{#2 \ref{##1}}%
		\begin{rep@theorem}}%
		{\end{rep@theorem}}}
\makeatother

\newreptheorem{Thm}{Theorem}
\newreptheorem{Cor}{Corollary}

\newcommand{\suchthat}{\;\ifnum\currentgrouptype=16 \middle\fi|\;}

\newcommand\restr[2]{{
		\left.\kern-\nulldelimiterspace 
		#1 
		\vphantom{\big|} 
		\right|_{#2} 
}}

\begin{document}

\maketitle

\begin{abstract}
Given an endomorphism $\varphi: G \to G$ on a group $G$, one can define the Reidemeister number $R(\varphi) \in \N \cup \{\infty\}$ as the number of twisted conjugacy classes. The corresponding Reidemeister zeta function $R_\varphi(z)$, by using the Reidemeister numbers $R(\varphi^n)$ of iterates $\varphi^n$ in order to define a power series, has been studied a lot in the literature, especially the question whether it is a rational function or not. For example, it has been shown that the answer is positive for finitely generated torsion-free virtually nilpotent groups, but negative in general for abelian groups that are not finitely generated.

However, in order to define the Reidemeister zeta function of an endomorphism $\varphi$, it is necessary that the Reidemeister numbers $R(\varphi^n)$ of all iterates $\varphi^n$ are finite. This puts restrictions, not only on the endomorphism $\varphi$, but also on the possible groups $G$ if $\varphi$ is assumed to be injective. In this note, we want to initiate the study of groups having a well-defined Reidemeister zeta function for a monomorphism $\varphi$, because of its importance for describing the behavior of Reidemeister zeta functions.  As a motivational example, we show that the Reidemeister zeta function is indeed rational on torsion-free virtually polycyclic groups. Finally, we give some partial results about the existence in the special case of automorphisms on finitely generated torsion-free nilpotent groups, showing that it is a restrictive condition.

\medskip


\end{abstract}

\section{Introduction}

Given a group $G$ and an endomorphism $\varphi: G \to G$ on the group $G$, we say that the $x, y \in G$ are $\varphi$-conjugate if and only if there exists a $z \in G$ such that $$x = z y \varphi(z)^{-1}$$ and write this as $x \sim_\varphi y$. This notion generalizes regular conjugacy, which corresponds to the case $\varphi = \Id_G$. It is immediate that $\sim_\varphi$ forms an equivalence relation on the group $G$, and the corresponding equivalence classes $[x]_\varphi = \left\{y \in G \mid y \sim_\varphi x \right\}$ are called $\varphi$-conjugacy or twisted conjugacy classes.  The number of equivalence classes is called the \textbf{Reidemeister number} of $\varphi$ and denoted as $$R(\varphi) \in \N \cup \{\infty\}.$$ 

Of course, if $\varphi$ is an endomorphism of the group $G$, then so is every iterate $\varphi^n$, and thus we get for every $n  > 0$ a Reidemeister number $R(\varphi^n) \in\N \cup \{\infty\}$. If all these numbers happen to be finite, we can define the \textbf{Reidemeister zeta function} $R_\varphi(z)$ as $$R_\varphi(z) = \exp \left( \sum_{n=1}^\infty 
\frac{R(\varphi^n) }{n} z^n\right).$$
Ever since the Reidemeister zeta function was defined in \cite{fels91-1}, mathematicians have tried to show that it is \textbf{rational} for several classes of groups. We give an overview of known results below, more details can be found in \cite{fels00-2}. 

If the group $G$ is finite, a direct sum of a finite group and a finitely generated free abelian group or a finitely generated torsion-free nilpotent group, it is known that the Reidemeister zeta function is always rational. Moreover, for any finitely generated group $G$ and any endomorphism $\varphi$ that is eventually commutative, the same holds. Later, by using the relation to the Nielsen zeta function, the case of rationality for almost-Bieberbach groups, i.e.~finitely generated torsion-free virtually nilpotent groups, was done by \cite{dd13-2}. Finally, for automorphisms on almost-crystallographic groups, it is shown in \cite{dtv18-1} that up to Hirsch length $3$ every Reidemeister zeta function is rational, together with some more general statements for certain holonomy representations. 

All results mentioned above lead to the following conjecture due to Fel'shtyn about the rationality of the Reidemeister zeta function.

\begin{Con}
	\label{con}
For all endomorphisms $\varphi: \Gamma \to \Gamma$ on a virtually polycyclic group $\Gamma$, the Reidemeister zeta function $R_\varphi(z)$ is always rational whenever it is well-defined.
\end{Con}

However, there are several instances for which the Reidemeister zeta function is \textbf{not rational}. Indeed, \cite[Lemma 5.9.]{fm20-1} shows that on the abelian group $\Z\left[\frac{1}{3}\right]$, the endomorphism $\varphi(x) = 2x$ has a non-rational zeta function. Similarly \cite[Section 1.2.]{bfm22-1} contains such an example on an abelian torsion group that is not finitely generated. In fact, recent papers show that in all known examples, there is a so-called P\'olya-Carlson dichotomoy for Reidemeister zeta functions, meaning that either $R_\varphi(z)$ is rational or has a natural boundary at its radius of convergence, see \cite{bfm22-1,ft21-1,fk22-1,fm20-1}. 

The main motivation for studying the Reidemeister number and the corresponding zeta function lies in \textbf{fixed point theory}. Indeed, given a self-map $f: M \to M$ on a manifold, one can write the set of fixed points as a disjoint union of fixed point classes. The number of these fixed point classes is defined as the Reidemeister number of $f$, and it is in fact equal to the Reidemeister number of the induced morphism $f_\ast$ on the fundamental group. The fixed point classes are used to define the Nielsen number, a homotopy invariant lower bound for the number of fixed points of $f$. For more details on this point of view, we refer to \cite{fels00-2}.

Note that, in order to define the Reidemeister zeta function above, one needs that the iterates of $\varphi$ have a finite Reidemeister number. 
\begin{Def}
Let $G$ be any group, then we call an endomorphism $\varphi: G \to G$ \textbf{tame} if all iterates have a finite Reidemeister number, i.e.~if $R(\varphi^n) < \infty$ for all $n > 0 $. 
\end{Def}
\noindent Equivalently an endomorphism $\varphi: G \to G$ is tame if and only it has a well-defined Reidemeister zeta function $R_\varphi(z)$. By considering the trivial endomorphism $\varphi$ defined as $\varphi(g) = e$ for all $g \in G$, for which $R(\varphi^n) = 1$ for all $n > 0$, we see that every group $G$ has at least one tame endomorphism. 

The \textbf{existence of a tame automorphism} or even just a monomorphism (i.e.~endomorphisms that are injective) on $G$ is a non-trivial condition. For example in \cite[Section 8.]{fl14-1}, the authors show that if a homeomorphism on an infra-solvmanifold of type $(R)$ has a well-defined Reidemeister zeta function, or equivalently the corresponding fundamental group has a tame automorphism, then the manifold must be an infra-nilmanifold, or thus that the group is virtually nilpotent. This was later generalized in \cite{dere22-1} by showing that every virtually polycyclic group with a tame monomorphism must be virtually nilpotent. This motivates the following general question for groups.

\begin{QN}
	\label{QN1}
Which groups $G$ have an automorphism $\varphi: G \to G$ (or more generally a monomorphism $\varphi: G \to G$) that is tame?
\end{QN}
\noindent We will argue why monomorphisms are the good class to consider in Section \ref{sec:poly}.

The reason for our intrest in Question \ref{QN1} is the fact that it also sheds light on the rationality of the Reidemeister zeta function. Indeed, answers to Question \ref{QN1} help to reduce the computation of these function to cases that are already known. This is also the strategy in \cite[Section 3.]{dtv18-1}, where the authors prove how the existence of a Reidemeister zeta function for an automorphism on the group influences the structure of the almost crystallographic group, for example in relation to the holonomy representation. In Section \ref{sec:poly} we motivate how Question \ref{QN1} helps to study the behavior of the Reidemeister zeta function, by proving the following.

\begin{theorem}\label{thm:torsionfree}
	Let $\varphi: \Gamma \to \Gamma$ be a tame endomorphism of a torsion-free virtually polycyclic group $\Gamma$, then the corresponding Reidemeister zeta function $R_\varphi$ is rational.
\end{theorem}


Afterwards, in Section \ref{sec:nil}, we initiate the study of Question \ref{QN1} for the special case of automorphisms on finitely generated torsion-free nilpotent groups, from now on called \TG s. We say that a group $G$ has the $R_\infty$-property if every automorphism has infinite Reidemeister number. Of course, groups with the  $R_\infty$-property do not admit a tame automorphism, but the converse is far from true, as we illustrate with some examples. Our results show that having a tame automorphism is preserved under commensurability, in contrast to the $R_\infty$-property. 

On the other hand, having a tame automorphism is not preserved by changing from one cocompact lattice in a $1$-connected nilpotent Lie group to another, so in particular, it is not a quasi-isometric invariant. We give several examples that show that the existence of a tame automorphism is closely related to the existence of an Anosov automorphism, making it a very restrictive condition on \TG s. Finally, we conclude this note with several questions that arise from the results and examples presented below.

\section{Reidemeister zeta function on virtually polycyclic groups}
\label{sec:poly}

In this section, we focus on the case of virtually polycyclic groups, to motivate the importance of Question \ref{QN1}. The first part combines the results of \cite{fm20-1} with the ones in \cite{dere22-1} to relate the Reidemeister zeta function on these groups to the ones on virtually nilpotent groups. We make this more concrete in the second part by showing that the Reidemeister zeta function is rational for the torsion-free groups in this class, by studying the invariant subgroups in more detail.

Recall that a \textbf{polycyclic group} $\Gamma$ is one that has by definition a subnormal series $$\{e\}= \Gamma_0 \triangleleft \Gamma_1 \triangleleft \ldots \triangleleft \Gamma_k = \Gamma$$ such that $\Gamma_{i+1}/\Gamma_i$ is cyclic for all $i= 0, \ldots, k-1$. Equivalently, a polycyclic group is a solvable group for which all subgroups are finitely generated. A virtually polycyclic group is a group containing a polycyclic group of finite index, and thus the class of virtually polycyclic group is closed under taking extensions, quotients and subgroups. Moreover, every virtually polycyclic group has a maximal torsion normal subgroup $T$, and as $T$ is virtually solvable, finitely generated and consists of torsion elements, it must be finite. Every finitely generated virtually nilpotent group is virtually polycyclic, and thus satisfies the same properties.

\subsection{Reidemeister number of quotients and subgroups} 
As mentioned in the introduction, we restrict our attention to monomorphisms in Question \ref{QN1} for finding useful restrictions on the group $G$. In order to deduce information about general morphisms we introduce two results of \cite{fm20-1}, allowing to take certain quotients or restrict to certain subgroups. The first tool is \cite[Lemma 4.1.]{fm20-1} dealing with invariant subgroups.
\begin{Prop}
	\label{prop:sub}
	Let $\varphi: G \to G$ be an endomorphism of a group $G$, and let $H$ be any subgroup of $G$ such that 
	\begin{enumerate}[(1)]
		\item $H$ is invariant under $\varphi$, i.e.~$\varphi(H) \subset H$;
		\item for all $x \in G$, there exists some $n > 0$ such that $\varphi^n(x) \in H$.
	\end{enumerate}
	If we denote by $\varphi_H: H \to H$ the restriction of $\varphi$ to $H$, then $R(\varphi) = R(\varphi_H)$. In particular, $\varphi$ is tame if and only if $\varphi_H$ is tame and in that case we have $$R_\varphi(z) = R_{\varphi_H}(z).$$
\end{Prop}

One peculiar class of examples on which we can apply Proposition \ref{prop:sub} are the subgroups $H_n = \varphi^n(G) \subset G$ for $n > 0$. However, in general it cannot be applied to the subgroup $\displaystyle \bigcap_{n>0}\varphi^n(G)$, as it does not automatically satisfy the second condition. For example, the map $\varphi: \Z \to \Z$ given by $\varphi(x) = 2x$ has $\displaystyle \bigcap_{n > 0} \varphi^n(G) = \{e\}$, but as for every $n > 0$ the map $\varphi^n$ is injective, there is no $n > 0$ such that $\varphi^n(x) = e$ for $x \neq e$.

The second tool is an adaption from \cite[Theorem 4.5]{fm20-1}. We present it in this more general form, as this way it resembles the previous proposition. 
\begin{Prop}
	\label{prop:quot}
	Let $\varphi: G \to G$ be an endomorphism of a group $G$, and let $N$ be any normal subgroup such that \begin{enumerate}[(1)]
		\item $N$ is invariant under $\varphi$, i.e.~$\varphi(N) \subset N$;
		\item for all $x \in N$, there exists some $n > 0$ with $\varphi^n(x) = e$.
	\end{enumerate} Then $\varphi$ induces an endomorphism $\overline{\varphi}: G/N \to G/N$ with $R(\varphi) = R(\overline{\varphi})$. In particular, $\varphi$ is tame if and only if $\overline{\varphi}$ is tame, and in that case we have $$R_\varphi(z) = R_{\overline{\varphi}}(z).$$
\end{Prop}
For completeness, we give its proof, although it is only a minor adaptation from the original one, which deals with a special normal subgroup that satisfies the conditions of the proposition.  

\begin{proof}
	As the normal subgroup $N$ is invariant under $\varphi$, the induced map $\overline{\varphi}$ is indeed well-defined. If we write $\pi: G \to G/N$ for the projection map, then we will show that $\pi$ induces a bijection between the twisted conjugacy classes of $\varphi$ and $\overline{\varphi}$. 
	
	For every $x, y \in G$ with $x \sim_\varphi y$, we clearly have that $\pi(x) \sim_{\overline{\varphi}} \pi(y)$. Hence, it remains to show that if $\pi(x) \sim_{\overline{\varphi}} \pi(y)$, that $x \sim_\varphi y$. For this, let $\pi(z) \in G/N$ be an element such that $$\pi(x) = \pi(z) \pi(y) \overline{\varphi}(\pi(z))^{-1} =\pi(z) \pi(y) \pi(\varphi(z))^{-1}  .$$ This implies that $z^{-1} x \varphi(z) y^{-1} \in N$. In particular, there exists $n > 0$ such that $$\varphi^n(z^{-1} x \varphi(z) y^{-1}) = e,$$ or thus such that $\varphi^n(x) \sim_\varphi \varphi^n(y)$. As $x \sim_\varphi \varphi^n(x)$ and $y \sim_\varphi \varphi^n(y)$, we conclude that the proposition holds.
\end{proof}

Hence, every Reidemeister zeta function $R_\varphi(z)$ is equal to one induced by a monomorphism $\overline{\varphi}: G / N \to G/N$, by taking $$N = \{x \in \Gamma \mid \exists n > 0: \varphi^n(x)= e\} = \displaystyle \bigcup_{n > 0} \ker(\varphi^n),$$ which obviously satisfies the conditions of the proposition. This motivates why we only consider monomorphisms in Question \ref{QN1}.

However, on virtually polycyclic groups this has consequences for the group structure, see \cite[Theorem 3.12.]{dere22-1}.

\begin{theorem}
	\label{thm:virtnil}
Let $\varphi: \Gamma \to \Gamma$ be a tame monomorphism on a virtually polycyclic group, then $\Gamma$ must be virtually nilpotent.
\end{theorem}

Combining the two previous results immediately implies the following.

\begin{corollary}
	\label{cor:virtnilp}
	The possible Reidemeister zeta functions $R_\varphi(z)$ for tame endomorphisms $\varphi:\Gamma \to \Gamma$ on virtually polycylic groups $\Gamma$ are exactly the possible Reidemeister zeta functions $R_{\varphi^\prime}(z)$ of tame monomorphisms $\varphi^\prime: \Gamma^\prime \to \Gamma^\prime$ on finitely generated virtually nilpotent groups $\Gamma^\prime$.
\end{corollary}

\begin{proof}
Note that every finitely generated virtually nilpotent group is automatically virtually polycyclic, so it is clear that the Reidemeister zeta functions on virtually polycyclic groups contain the ones from monomorphisms on finitely generated virtually nilpotent groups. 

To show that we find all Reidemeister zeta functions in this way, we start from a tame endomorphism $\varphi: \Gamma \to \Gamma$ on a virtually polycyclic group. By applying Proposition \ref{prop:quot} to the normal subgroup $N = \{x \in \Gamma \mid \exists n > 0: \varphi^n(x)= e\}$ as before, we find that $\overline{\varphi}: \Gamma/N \to \Gamma/N$ is a tame monomorphism and $R_\varphi(z) = R_{\overline{\varphi}}(z)$. Moreover, Theorem \ref{thm:virtnil} shows that the group $\Gamma/N$ must be virtually nilpotent, leading to the desired result.
\end{proof}

This shows that for Conjecture \ref{con} the focus should be on computing the Reidemeister zeta function on virtually nilpotent groups. As we will show in Section \ref{sec:nil}, having a tame automorphism is a restrictive condition on \TG s, reducing the number of groups for which we have to compute $R_\varphi(z)$. However, a more careful analysis using invariant subgroups allows us to deduce even more information of the known cases, as demonstrated in the next section.

\subsection{Structure of tame endomorphisms}

If $\Gamma$ is a polycyclic group with subnormal series $\Gamma_i \subset \Gamma$ for $i \in \{0, \ldots, k\}$ as in the definition, then we define the \textbf{Hirsch length} $h(\Gamma)$ of the group $\Gamma$ as the number of infinite cyclic quotients $\Gamma_{i+1}/\Gamma_i$. A standard argument shows that this does not depend on the choice of subnormal series. Similarly, we can define the Hirsch length for virtually polycyclic groups as the Hirsch length of a finite index polycyclic subgroup, and again this does not depend on the choice of subgroup. The Hirsch length satisfies the following properties elementary properties, see \cite[Section 1C]{sega83-1}.

\begin{Prop}
	\label{prop:hirsch}
	Let $\Gamma$ be a virtually polycyclic group with subgroup $\Gamma^\prime \subset \Gamma$, then the following statements hold. 
	\begin{enumerate}[(1)]
		\item $h(\Gamma^\prime) \leq h(\Gamma)$, with $h(\Gamma^\prime) = h(\Gamma)$ if and only if $\Gamma^\prime$ has finite index in $\Gamma$.
		\item $h(\Gamma) = 0$ if and only if $\Gamma$ is finite. 
		\item If $\Gamma^\prime$ is normal in $\Gamma$, then $h(\Gamma) = h(\Gamma^\prime) + h(\Gamma/\Gamma^\prime)$
	\end{enumerate}
\end{Prop}

%

We also need the following notion for endomorphisms.
\begin{Def}
Let $\mathcal{P}$ be a property of groups. We call an endomorphism $\varphi: G \to G$ \textbf{eventually $\mathcal{P}$} if there exists an $n > 0$ such that $\varphi^n(G)$ has property $\mathcal{P}$.
\end{Def}
\noindent For example, an endomorphism is called eventually abelian if there exists an $n >0$ such that $\varphi^n(G)$ is abelian. These type of endomorphisms have been used a lot to determine rationality of the Reidemeister zeta function, see \cite{fels00-2}.

In the case of virtually polycyclic groups, we have the following consequence of tameness.
\begin{theorem}
	\label{thm:event}
If $\varphi$ is a tame endomorphism of a virtually polycyclic group, then it is eventually virtually nilpotent.
\end{theorem}

\begin{proof}
Let $\varphi$ be a tame endomorphism of a virtually polycyclic group $\Gamma$. Consider the sequence of subgroups $H_n = \varphi^n(\Gamma) \subset \Gamma$, satisfying $H_{n+1} \subset H_n$ and $\varphi(H_n) = H_{n+1}$. As $$h(H_{n+1}) = h(\varphi(H_{n})) \leq h(H_{n})$$ by Proposition \ref{prop:hirsch}, it must hold that there exists $n_1 > 0$ such that $h(H_{n_1}) = h(H_{n_1+n})$ for all $n \geq 0$. Note that, although this is not important for our proof, $n_1$ can be taken as the first value for which $h(H_{n_1}) = h(H_{n_1+1})$. 

We now take $H = H_{n_1}$ as the subgroup for Proposition \ref{prop:sub}. By construction we have that $\varphi_H^n(H) = H_{n_1+n} \subset H$, with both groups having the same Hirsch length. In particular, $\ker(\varphi_H^n)$ has Hirsch length $0$ for every $n > 0$ or thus is a finite normal subgroup of $H$. The normal subgroup $$N = \bigcup_{n> 0} \ker(\varphi_H^n)$$ hence lies completely in the maximal torsion normal subgroup $T$ of $H$, which is finite for virtually polycyclic groups. In particular, there exists $n_2$ such that $N = \ker(\varphi^{n_2}_H)$.  

Write $H^\prime = \varphi^{n_1+n_2}(\Gamma) = \varphi^{n_2}(H)$. It is now easy to see that $\varphi_{H^\prime}$ is a tame monomorphism on the subgroup $H^\prime$, and thus that $H^\prime$ is virtually nilpotent by Theorem \ref{thm:virtnil}. We conclude that $\varphi$ is eventually virtually nilpotent.
\end{proof}
\noindent From the proof it follows that there even exists $n>0$ such that $H = \varphi^n(\Gamma)$ is virtually nilpotent and $\varphi_H$ is injective. The advantage compared to Corollary \ref{cor:virtnilp} is that it is often easier to analyze the possible subgroups of a given group $\Gamma$ than all its quotients.

The converse of Theorem \ref{thm:event} is clearly not true, as we illustrate with the following example. 

\begin{Ex}
Consider the torsion-free polycyclic group $\Gamma = \Z^2 \rtimes_A \Z$ with matrix 
$$ A = \begin{pmatrix}
	2 & 1 \\ 1 & 1 \end{pmatrix}.$$ 
It is clear that the endomorphism $\varphi: \Gamma \to \Gamma$ defined by $\varphi(z,t) = (0,t)$ for $z\in \Z^2, t \in \Z$ is eventually virtually nilpotent, even eventually abelian, as $$H = \varphi(\Gamma) = \langle ((0,0),1) \rangle \approx \Z$$ is abelian. Moreover, $\varphi_H = \Id_H$ and thus $R(\varphi) = R(\varphi_H) = \infty$, meaning the endomorphism $\varphi$ is not tame.
\end{Ex}

The proof of Theorem \ref{thm:torsionfree} is now a consequence of the previous.

\begin{corollary}
	\label{cor:torsionfree}
	Let $\varphi$ be a tame endomorphism on a torsion-free virtually polycyclic group $\Gamma$, then the Reidemeister zeta function $R_\varphi(z)$ is rational.
\end{corollary}

\begin{proof}
By Theorem \ref{thm:event}, we know that $\varphi$ is eventually virtually nilpotent, so there exists $n > 0$ such that $H = \varphi^n(\Gamma)$ is virtually nilpotent. Moreover $H$ is torsion-free as a subgroup of $\Gamma$. As $R_\varphi(z) = R_{\varphi_H}(z)$, it must be rational as well by \cite{dd13-2}. 
\end{proof}

Using a similar method as in Corollary \ref{thm:torsionfree}, one could make statements using the Hirsch length of the group $\Gamma$. For example, if it would be known that all virtually nilpotent groups up to Hirsch length $3$ have a rational Reidemeister zeta function, a consequence would be that all virtually polycylic groups up to Hirsch length $4$ that are not virtually nilpotent have a rational Reidemeister zeta function. However, the first statement is not yet known in full generality, as \cite{dtv18-1} only deals with automorphisms on almost-crystallographic groups up to Hirsch length $3$.

\begin{Rmk}
As a final remark in this section, we want to make the connection with Nielsen zeta functions for self-maps on manifolds, and results on their rationality as in \cite{dv21-1}. As mentioned in the introduction, to every self-map $f: M \to M$ one can associated a number $N(f) \in \N$, called the Nielsen number. As this number is always finite, we can define for any self-map a Nielsen zeta function, just as we did for the Reidemeister zeta function. On the other hand, these fundamental groups are always torsion-free, and thus the class of groups to consider is smaller.
	
	In the case of infra-solvmanifolds, a class of manifolds having a torsion-free and virtually polycyclic fundamental group, it holds that $N(f) = R(f)$ if $R(f) < \infty$, see for example \cite[Proposition 3.1.3.]{dv21-1}. In particular, if $f$ induces a tame endomorphism on the fundamental group, the Nielsen zeta function and the Reidemeister zeta function are the same and thus rational via Theorem \ref{thm:torsionfree}. However, this does not imply that every Nielsen zeta function is rational, as not every endomorphism is tame.	
\end{Rmk}

\section{Tame automorphisms on \TG s}

\label{sec:nil}

As explained in the previous section, understanding the groups that admit a tame monomorphism is important to understand the behavior of Reidemeister zeta functions. In order to prove Conjecture \ref{con}, you need a full description in the case of virtually nilpotent groups by Corollary \ref{cor:virtnilp}. In this section we initiate this study for the special case of tame automorphisms on \TG s, which were defined as follows in the introduction.
\begin{Def}
A  \textbf{\TG}~is a finitely generated torsion-free nilpotent group.
\end{Def}
\noindent We will in fact show that only few \TG s have a tame automorphism.

We first introduce rational and real Mal'cev completions, and then prove how tame automorphisms depend on these completions. Afterwards, we compute some examples to show that in many cases, the existence of a tame automorphism is equivalent to the existence of an Anosov automorphism. We end the section with several open questions that arise from this note.

\subsection{Mal'cev completions}

One of the main tools for studying Reidemeister numbers on \TG s are the so-called Mal'cev completions, which exist for any field $F$ of characteristic $0$. However, for our purposes, $F$ will be either the field of rational numbers $\Q$ or the field of real numbers $\R$. There are different constructions, all leading to the same class of groups, and we refer to \cite{deki18-1} for more details about the construction we consider below.

Given a \TG~$N$, then it can be embedded as a cocompact lattice into a simply connected and connected (abbreviated as $1$-connected) nilpotent Lie group $N^\R$. The Lie group $N^\R$ is called the  \textbf{real Mal'cev completion} of $N$. Let $\lien^\R$ be the corresponding Lie algebra, then it is well-known that the exponential map $\exp: \lien^\R \to N^\R$ is a bijection, and thus we can define its inverse function, namely $\log: N^\R \to \lien^\R$. Consider the subset $\log(N) \subset \lien^\R$, then by consequence of the Baker-Campbell-Hausdorff formula the rational span $\Q \log(N) = \lien^\Q$ forms a Lie algebra over the rational numbers. The corresponding group $N^\Q = \exp(\lien^\Q)$ is called the \textbf{rational Mal'cev completion} of $N$.

Every monomorphism $\varphi: N \to N$ uniquely extends to an automorphism of $N^F$, with $F$ either $\Q$ or $\R$ from now on. Moreover, this shows that the rational and real Mal'cev completion $N^\Q$ and $N^\R$ are unique up to isomorphism. The automorphisms of $\lien^F$ as a Lie algebra correspond under the exponential map to the automorphisms of $N^F$ as a group, and from now on we will identify these maps. Hence it makes sense to talk about the characteristic polynomial or eigenvalues of an automorphism on $N^F$ by using this correspondence.  Similarly, we can define these notions for an automorphism $\varphi: N \to N$ using the unique extension to  $N^F$, where the choice of field $F$ does not influence the outcome. As this will be convenient later, we consider eigenvalues as the (complex) roots of the characteristic polynomial, counted with their algebraic multiplicity, even if they are not elements of the field $F$ we are working with. 

The following result of \cite{roma11-1}, although written in a slightly different context, is useful to determine whether an automorphism $\varphi \in \Aut(N)$ has infinite Reidemeister number.
\begin{Prop}
	\label{prop:eig}
	Let $\varphi: N \to N$ be an automorphism on a \TG~$N$, and $\varphi^F$ the unique extension to the Mal'cev completion over $F$, where $F = \Q$ or $F= \R$. Then $R(\varphi) = \infty$ if and only if $\varphi^F$ has $1$ as eigenvalue.
\end{Prop}
\noindent The original proof of \cite{roma11-1} uses the torsion-free quotients of a central series of the nilpotent group $N$, but the eigenvalues correspond to our definition. Note that there are also formulas to compute $R(\varphi)$ from the eigenvalues of $\varphi^F$, but these are not important to us.

In particular, if an automorphism $\varphi \in \Aut(N)$ has no eigenvalues of absolute value $1$, then $R(\varphi) < \infty$. As this property on the eigenvalues is preserved under taking powers $\varphi^n$, we get that $R(\varphi^n) < \infty$ for all $n > 0$ or thus that $\varphi$ is tame. Such an automorphism is called \textbf{Anosov}, because it is closely related to the study of Anosov diffeomorphisms on closed manifolds, see \cite{deki99-1}.

\begin{corollary}
	\label{cor:anotame}
	Every Anosov automorphism on a \TG~$N$ is tame.
\end{corollary}

Note that being Anosov is stronger than being tame, as Proposition \ref{prop:eig} shows that the latter is equivalent to the weaker condition of having no roots of unity as eigenvalues. As we will illustrate in the final part of this section, there is no known example of a \TG~$N$ without Anosov automorphism but with a tame automorphism.

\subsection{Commensurable groups}

In this section, $N$ will denote a \TG~and $N^\Q$ its rational Mal'cev completion with Lie algebra $\lien^\Q$. We will show that the existence of a tame automorphism on $N$ only depends on $N^\Q$. 

The rational Mal'cev completion also has a completely algebraic description. It is the unique torsion-free radicable group $N^\Q$ containing $N$ as a subgroup such that each element $x \in N^\Q$ has some power $x^n$ with $n>0$ such that $x^n \in N$. Here, radicable means that for every element $x \in N^\Q$ and every $n > 0$, there exists a (necessarily unique for torsion-free nilpotent groups) $y \in N^\Q$ with $y^n = x$. We can now introduce the following concept.

\begin{Def}
We call a finitely generated subgroup $M \subset N^\Q$ a \textbf{full subgroup} if $N^\Q$ is also the rational Mal'cev completion of $M$. 
\end{Def}

\noindent It is immediate that every finitely generated subgroup $M$ of $N^\Q$ is a \TG, so the rational Mal'cev completion of $M$ is indeed defined. As $N^\Q$ is already a torsion-free radicable group containing $M$, the only condition to be checked is that every element $x \in N^\Q$ has some positive power $n > 0$ such that $x^n \in M$. 

Recall that two groups $G_1, G_2$ are called (abstractly) commensurable if and only if there exist finite index subgroups $G_1^\prime \subset G_1, G_2^\prime \subset G_2$ such that $G_1^\prime \approx G_2^\prime$. Using this relation, we can describe the algebraic structure of full subgroups of $N^\Q$.

\begin{theorem}
	Let $M$ and  and  be \TG s, then the groups $M$ and $N$ are commensurable if and only if the rational Mal'cev completions are isomorphic, i.e.~$M^\Q \approx N^\Q$. 
\end{theorem}

\noindent So up to isomorphism, full subgroups of $N^\Q$ are exactly the \TG s commensurable to $N$. The following result of \cite{deki99-1} relates the automorphisms of $N$ to that of a full subgroup $M \subset N^\Q$. 
\begin{theorem}
	\label{thm:power}
	Let $N$ be a \TG~with rational Mal'cev completion $N^\Q$ and $M$ a full subgroup of $N^\Q$. For every automorphism $\psi: M \to M$ with extension $\psi^\Q: N^\Q \to N^\Q$, there exists $k > 0$ such that $(\psi^\Q)^k(N) = N$, or thus such that $(\psi^\Q)^k$ induces an automorphism on $N$.
\end{theorem}

Using Proposition \ref{prop:eig} above, this leads to the following consequence.

\begin{theorem}
	\label{thm:comm}
	Let $M$ and $N$ be \TG s that are commensurable. Then the group $M$ has a tame automorphism if and only if the group $N$ has a tame automorphism. 
\end{theorem}

\begin{proof}
	By symmetry, it suffices to show that if $M$ has a tame automorphism $\psi$, then also $N$ has a tame automorphism. As $M$ and $N$ are commensurable, they have an isomorphic rational Mal'cev completion. Hence we can assume without loss of generality that $M$ is a full subgroup of $N^\Q$. 
	
	If $\psi^\Q$ is the extension of $\psi$ to $N^\Q$, then $\psi^n$ has extension $\left(\psi^\Q\right)^n$ on $N^\Q$. By Theorem \ref{thm:power}, $\psi^\Q$ has a power $k$ such that $(\psi^\Q)^k(N) = N$, and thus induces an automorphism $\varphi$ on $N$. As the eigenvalues of $\varphi^n$ and $\psi^{kn}$ correspond as they have the same extension to $N^\Q$, we immediately get that $\varphi$ is tame.
\end{proof}

This result is in contrast to the $R_\infty$-property, which is not preserved for commensurable groups, even not in the special case of \TG s. Indeed, recent work \cite{lw23-1} gives examples of two \TG s $M$ and $N$ of nilpotency class $2$ that are commensurable, but one has $R_\infty$ and the other does not. Note that Theorem \ref{thm:comm} does imply that both groups from \cite{lw23-1} do not have a tame automorphism.

In fact, we can fully characterize the existence of a tame automorphism on the rational Mal'cev completion. An automorphism is called \textbf{integer-like} if it has characteristic polynomial in $\Z[X]$ and determinant $\pm1$, where we recall that these properties are defined via the Lie algebra. 

\begin{theorem}
	\label{thm:char}
A \TG~$N$ admits a tame automorphism if and only if $N^\Q$ admits an integer-like automorphism with no eigenvalues that are roots of unity.
\end{theorem}

\begin{proof}
Note that not every automorphism of $N^\Q$ induces an automorphism on $N$. The exact correspondence has been developed in \cite[Theorem 6.1.]{dere15-1}, where it is shown that $\psi \in \Aut(N^\Q)$ induces an automorphism on some full subgroup $M \subset N^\Q$ if and only if $\psi$ is integer-like. Note that if $\psi$ has no eigenvalues that are roots of unity, then $\psi^n$ for $n > 0$ has no eigenvalues equal to $1$. In particular, $\psi$ induces a tame automorphism on $M$. As $N$ and $M$ are commensurable, the previous theorem implies that $N$ admits a tame automorphism.

For the other direction, assume that $\varphi: N \to N$ is an automorphism. Take $M$ equal to the lattice hull of $N$, i.e.~the smallest subgroup of $N^\Q$ such that $\log(M)$ is closed under addition. By \cite{bele03-1} it is known that $N$ is finite index in $M$ and moreover $\varphi^\Q(M) = M$. The corresponding Lie algebra automorphism maps $\log(M) \approx \Z^k$ onto itself, or thus can be represented by an element in $\GL(k,\Z)$. We conclude that the map $\varphi$ must be integer-like. On the other hand, if $\varphi$ has an eigenvalue $\lambda$ such that $\lambda^n =1$ for some $n \geq 1$, then Proposition \ref{prop:eig} implies that $\varphi^n$ satisfies $R(\varphi^n) = \infty$. Hence, if $\varphi$ is tame, then $\varphi$ has no eigenvalues that are roots of unity, completing the other direction.
\end{proof}

\noindent We call a rational nilpotent Lie algebra tame if it admits an integer-like automorphism without roots of unity as eigenvalues, similarly as the definition of Anosov Lie algebras.

This result is equivalent to the one for Anosov diffeomorphisms \cite{deki99-1} and for expanding maps \cite{dd14-1}. However, in the latter case, it is moreover true that it only depends on the real Mal'cev completion. In the next subsection we investigate whether the same holds for tame automorphisms.

\subsection{Lattices in Lie groups}

As mentioned above, every \TG~$N$ has a corresponding real Mal'cev completion $N^\R$, which contains $N$ as a discrete and cocompact subgroup. The following result is a rephrasing of \cite[Theorem 3.2.]{lw08-1}.
\begin{Prop}
	\label{lauret}
Let $N^\R$ be a real Mal'cev completion of a \TG~$N$ such that the corresponding Lie algebra has a positive grading. The direct sum $G = \displaystyle \bigoplus_{i=1}^s N^\R$ with $s \geq 2$ contains a cocompact lattice $N^\prime \subset G$ that has an Anosov automorphism.
\end{Prop}

In particular, the group $N^\prime$ is a \TG~that has a tame automorphism by Corollary \ref{cor:anotame}. We show that sometimes the same Lie group contains different lattices with possibly no tame automorphisms. For this, we first need to know more about automorphisms of direct sums of Lie algebras.

Recall that a Lie algebra $\lie$ over a field $F$ is \textbf{indecomposable} if it has no decomposition $\lie = \lie_1 \oplus \lie_2$ into two non-trivial ideals $\lie_i$ of $\lie$. Every finite dimensional Lie algebra $\lie$ can be written as a direct sum of indecomposable ideals $\lie =  \displaystyle \bigoplus_{i=1}^s \lie_i$. If one of the Lie algebras $\lie_i$ is abelian, we say that $\lie$ has an abelian factor, and the same notion will be used for rational and real Mal'cev completions via their Lie algebra. Given automorphisms $\psi_i: \lie_i \to \lie_i$, we naturally have an automorphism on $\lie$ given by $$(\psi_1, \ldots, \psi_s): \bigoplus_{i=1}^s \lie_i \to \bigoplus_{i=1}^s \lie_i$$ which maps each $x_i \in \lie_i$ to $\psi_i(x_i) \in \lie_i$. The following proposition states that any automorphism resembles one of these automorphisms in a specific way.

\begin{Prop}
	\label{prop:preserve}
Let $\lie =\displaystyle \bigoplus_{i=1}^s \lie_i$ be a Lie algebra over a field $F$, written as a direct sum of ideals $\lie_i$, where every $\lie_i$ is indecomposable and non-abelian. Then any automorphism $\psi: \lie \to \lie$ has some power $\psi^k$ with $k > 0$ such that $\psi^k$ has the same characteristic polynomial as $(\psi_1,\ldots,\psi_s)$ for some $\psi_i \in \Aut(\lie_i)$. 
\end{Prop}

\begin{proof}
This is a special case of \cite[Theorem 4.5]{dw22-1} with the abelian factor $\mathfrak{a}$ trivial. 
\end{proof}

We also need the following characterization of integer-like automorphisms. Recall that an element $\alpha \in \C$ is an \textbf{algebraic integer} if it is algebraic over $\Q$ and its minimal polynomial over $\Q$ has coefficients over $\Z$. A rational number $\alpha \in \Q$ is an algebraic integer if and only if $\alpha \in \Z$. If the characteristic polynomial of an algebraic integer $\alpha \in \C$ moreover has constant term $\pm 1$, we call $\alpha$ an \textbf{algebraic unit}. In particular, the only algebraic units in $\Q$ are $\pm 1$. The notion of algebraic units resembles the definition of integer-like matrices, which is not a coincidence.

\begin{Lem}
	\label{lem:int}
An automorphism $\psi \in \Aut(\lien^\Q)$ of a rational Lie algebra is integer-like if and only if its eigenvalues are algebraic units. 
\end{Lem}

\begin{proof}
Let $\psi$ be an integer-like automorphism. If $\alpha \in \C$ is a root of the characteristic polynomial, it is of course algebraic over $\Q$. Moreover, by Gauss's lemma its minimal polynomial over $\Q$ lies in $\Z[X]$ and has constant term $\pm 1$.

Conversely, if every root of the characteristic polynomial is an algebraic unit, then so is the constant term of the characteristic polynomial as a product of algebraic units. Moreover, as it lies in $\Q$, it must be $\pm1$. Similarly, the other coefficients of the characteristic polynomial are algebraic integers as products and sums of algebraic integers, and thus lie in $\Z$.
\end{proof}

This immediately implies the following for automorphisms on the rational Mal'cev completions of $\Z$ and $H_3(\Z)$.
\begin{Lem}
	\label{lem:heis}
Every integer-like automorphism of either the additive group $\Q$ or the rational Heisenberg group $H_3(\Q)$ has an eigenvalue $\pm1$.
\end{Lem}
\begin{proof}	
	For $\Q$ this is clear, as the only integer-like automorphisms are multiplication with $\pm1$. For $H_3(\Q)$, we prove this at the level of the Lie algebras, where we write $\mathfrak{h}^\Q$ for the corresponding rational Lie algebra. As the center of $\mathfrak{h}^\Q$ has dimension $1$, it corresponds to a rational eigenvalue. This eigenvalue is also an algebraic unit by Lemma \ref{lem:int} and thus must be $\pm1$.
\end{proof}

We are now ready to prove the main result of this subsection.
\begin{theorem}
	\label{thm:noQI}
	There exist a $1$-connected nilpotent Lie group $G$ with two cocompact lattices $N_1$ and $N_2$, such that $N_1$ admits a tame automorphism and $N_2$ does not.
\end{theorem}

\begin{proof}
Let $N = H_3(\Z)$ be the discrete Heisenberg group, then we know that the Lie group $N^\R = H_3(\R)$ corresponds to a Lie algebra with a positive grading. Moreover, this Lie algebra is indecomposable and non-abelian. Consider the direct sum $G = \displaystyle \bigoplus_{i=1}^s N^\R$ with $s \geq 2$. By Proposition \ref{lauret} we know that there exists a cocompact lattice $N_1$ of $G$ with an Anosov automorphism, and in particular this automorphism is tame. 

However, also $N_2 = \displaystyle \bigoplus_{i=1}^s N$ is a cocompact lattice in $G$, and we claim it does not admit a tame automorphism. Indeed, let $\varphi$ be any automorphism with extension $\varphi^\Q$ to $N_2^\Q = \displaystyle \bigoplus_{i=1}^s N^\Q$. By Proposition \ref{prop:preserve} it has some iterate $\varphi^k$ such that its extension $\left(\varphi^\Q\right)^k$ has the same characteristic polynomial as $ (\psi_1,\ldots,\psi_s)$ with $\psi_i \in \Aut(H_3(\Q))$. Note that the map $\left(\varphi^\Q\right)^k$ is integer-like as extension of an automorphism on $N_2$. By Lemma \ref{lem:int} this is equivalent to the fact that all its eigenvalues are algebraic units. Hence, the maps $\psi_i$ only have eigenvalues that are algebraic units and thus must be integer-like as well. In particular, each $\psi_i^2$ has eigenvalue $1$ by Lemma \ref{lem:heis}, and thus $\varphi^{2k}$ also has eigenvalue $1$.
\end{proof}
\noindent In fact we can replace $H_3(\Z)$ in the proof with any group such that the rational Mal'cev completion has no abelian factor, satisfies Lemma \ref{lem:heis} and has a positive grading on the Lie algebra.

 It is well-known that different cocompact lattices $N_1, N_2$ in the same Lie group $G$ are quasi-isometric, hence we have the following consequence.
\begin{corollary}
The existence of a tame automorphism is not a quasi-isometric invariant on finitely generated groups.
\end{corollary}

By following the same proof as in Theorem \ref{thm:noQI}, one has the following consequence. We only sketch the proof, as it is almost identical as the one given above.

\begin{theorem}
Let $M$ and $N$ be \TG s such that $M^\Q$ and $N^\Q$ do not have an abelian factor. Then $M \oplus N$ has a tame automorphism if and only if both $M$ and $N$ have a tame automorphism.
\end{theorem}

\begin{proof}
One direction is clear, as tame automorphisms $\varphi_1 \in \Aut(M), \varphi_2 \in \Aut(N)$ can be combined into the tame automorphism $(\varphi_1,\varphi_2)$ on $M \oplus N$. 

For the other direction, let $M^\Q = \displaystyle \bigoplus_{i=1}^s M_i^\Q$ and $N^\Q = \displaystyle \bigoplus_{j=1}^t N_j^\Q$ be the decomposition into indecomposable groups, which follows from the decomposition of the rational Lie algebra. If one starts from a tame automorphism $\varphi \in \Aut(M \oplus N)$, one finds just as in the proof of Theorem \ref{thm:noQI} that there exists integer-like automorphisms $\psi_i \in \Aut(M_i^\Q), \psi^\prime_j \in \Aut(N_j^\Q)$ with no eigenvalues that are roots of unity. In particular, we can conclude via Theorem \ref{thm:char} that both $M$ and $N$ have a tame automorphism, again by considering the automorphisms $(\psi_1, \ldots, \psi_s)$ and $(\psi^\prime_1, \ldots, \psi^\prime_t)$ on $M^\Q$ and $N^\Q$.
\end{proof}

\subsection{Characterization for certain classes}

In this section, we characterize for different families of groups which ones admit a tame automorphism. We will relate this to other properties of the group, such as $R_\infty$ or having an Anosov automorphism.

It is clear that if a group has the $R_\infty$-property, it does not admit a tame automorphism. However, the converse is not true, the easiest examples being $\Z$ and $H_3(\Z)$. 

\begin{Ex}
	\label{ex:heis}
The only automorphisms of $\Z$ are $\Id_\Z$ and $-\Id_\Z$. As $R(\Id_\Z) = \infty$ and $(-Id_\Z)^2 = \Id_\Z$, it is clear that $R(\varphi^2) = \infty$ for any $\varphi \in \Aut(\Z)$. 
	
Let $H_3(\Z)$ the discrete Heisenberg group. Every automorphism $\varphi: H_3(\Z) \to H_3(\Z)$ preserves the center $Z(H_3(\Z))$. As this is a cyclic group, generated by an element $c \in H_3(\Z)$, the automorphism $\varphi$ satisfies either $\varphi(c) = c$ or $\varphi(c) = c^{-1}$. In both case we have that $R(\varphi^2)=\infty$, for example by using Proposition \ref{prop:eig}, or thus $H_3(\Z)$ does not admit a tame  automorphism. 

These examples also illustrate Theorem \ref{thm:char}, as every integer-like automorphism indeed has an eigenvalue that is a root of unity by Lemma \ref{lem:heis}.
\end{Ex}

This generalizes to the class of free $c$-step nilpotent groups $N_{k,c}$ on $k$ generators, of which both groups above are an example, namely $\Z = N_{1,c}$ and $H_3(\Z) = N_{2,2}$.
\begin{Prop}
	\label{prop:free}
	Let $N_{k,c}$ be a free nilpotent group of nilpotency class $c$ on $k \geq 2$ generators. The group $N_{k,c}$ has a tame automorphism if and only if $c < k$. 
\end{Prop}

\begin{proof}
One direction is clear, as if $c < k$, the group $N_{k,c}$ has an Anosov automorphism, see \cite{dv09-1}. For the other direction, consider an automorphism $\varphi: N_{k,c} \to N_{k,c}$ and assume that $c \geq k$. Using the corresponding real Lie algebra $\lien^\R$ and the extension $\varphi^\R \in \Aut(\lien^\R)$, which is a free $c$-step nilpotent Lie algebra on $k$ generators, we will show that $\varphi^2$ has an eigenvalue $1$. The argument greatly uses the techniques as described in \cite{dg14-1}. 

Write $\overline{\varphi}^\R$ for the induced map on the abelianization $\faktor{\lien^\R}{[\lien^\R,\lien^\R]}$ and denote by $\lambda_1, \ldots, \lambda_k$ the (complex) eigenvalues of $\overline{\varphi}^\R$, counted with multiplicity. Then $\varphi^\R$ also has an eigenvalue equal to $\displaystyle \prod_{i=1}^k \lambda_i$ as $k \geq c$, see \cite[Lemma 2.4.]{dg14-1}. Moreover this is equal to the determinant of $\overline{\varphi}^\R$. As $\overline{\varphi}^\R$ is the extension of the induced map $\overline{\varphi}$ on $\faktor{N}{[N,N]}$, it is integer-like and thus this determinant is equal to $\pm1$. In particular, $\varphi^2$ has eigenvalue $1$ and thus $R(\varphi^2) = \infty$.
\end{proof}
\noindent Note the condition of Proposition \ref{prop:free} exactly corresponds to the condition for having an Anosov automorphism, see \cite{dv09-1}. In the following subsection we raise the question whether this is always the case.

There is another generalization possible of Example \ref{ex:heis}, namely the $2$-step nilpotent groups associated to a graph $X$, given by a set $V$ of vertices and a set $E$ of edges, see \cite{lw23-1}. We always assume that the graph $X$ is undirected and simple, i.e.~it has no loops or multiple edges. For example, on one vertex you can only have the addite group $\Z$, whereas on two vertices you can find either $H_3(\Z)$ or $\Z^2$, depending on whether the vertices are connected or not. 

Given such a graph $X$, we can define an equivalence relation $\sim$ on the vertices $V$, where $v \sim w$ if and only if the transposition $(vw)$ is a graph automorphism. The equivalence classes $\lambda_i$ for $i \in \{1,\ldots, s\}$ are called the coherent components components of $X$, with $V = \displaystyle \bigsqcup_{i=1}^s \lambda_i$. Tt is easy to see that either the subgraph $\lambda_i$ is empty (i.e.~has no edges) or is complete (i.e.~any two vertices in $\lambda_i$ are connected). If we write $N_{\lambda_i}^\Q$ for the rational Mal'cev completion of the group corresponding to this subgraph, then $N_X^\Q$ is generated by the subgroups $N_{\lambda_i}^\Q$ with $i=1, \ldots, s$. 

If $\psi_i^\Q$ are automorphisms of $N_{\lambda_i}^\Q$, it is not too hard to see that there exists an induced automorphism on $N_X^\Q$ given by $\psi_i$ on elements $x \in N_{\lambda_i}$. So in particular, we have a natural inclusion $\Aut(N_{\lambda_i}^\Q) \subset \Aut(N_X^\Q)$. In \cite{dm05-1,dm22-1}, the authors describe the structure of the $\Q$-linear algebraic group $\Aut(N_X^\Q)$ using these coherent components. We rephrase part of the main result for our purposes.

\begin{Prop}
Let $N_X^\Q$ be the rational Mal'cev completion of a $2$-step nilpotent group $N_X$ associated to a graph $X$. Every automorphism $\psi \in N_X^\Q$ has a power $\psi^k$ with $k > 0$ that has the same characteristic polynomial as an automorphism induced by automorphisms $\psi_i \in \Aut(N_{\lambda_i}^\Q)$.
\end{Prop}

\begin{proof}
We write $G^0$ for the connected component of the identity in a $\Q$-linear algebraic group $G$. Recall that a Levi subgroup is a maximal reductive subgroup in $G^0$, and these are unique up to conjugation. It is a consequence of \cite[Theorem 4.2.(ii)]{dm05-1} that if for every $i \in \{1, \ldots, s\}$ we have a Levi subgroup $L_i \subset \Aut(N_{\lambda_i}^\Q)^0$, then the product $L = \displaystyle \prod_{i=1}^s L_i$ forms a Levi subgroup of $\Aut(N_{X}^\Q)^0$.  As $\Aut(N_X^\Q)$ only has a finite number of connected components, there exists some $k > 0$ such that $\psi^k$ lies in $\Aut(N_X^\Q)^0$. The semisimple part of $\psi^k$ has the same characteristic polynomial as $\psi^k$ and moreover lies in some Levi subgroup of $\Aut(N_X^\Q)^0$. Hence it is conjugate in $\Aut(N_X^\Q)^0$ to an element of $L$, and hence to an element induced by automorphisms $\psi_i$.
\end{proof}
 
 We can now characterize the groups $N_X$ having a tame automorphism.
 
\begin{theorem}
	\label{thm:graph}
	Let $N_{X}$ be a $2$-step nilpotent group associated to a graph $X$. The group has a tame automorphism if and only if each coherent component $\lambda$ of $X$ either is empty of size $2$ or has size at least $3$.
\end{theorem}

\begin{proof}
	
	One direction is clear, as any group satisfying these properties has an Anosov automorphism by \cite{dm05-1}. For the other direction, assume that $\varphi: N_X \to N_X$ is a tame automorphism, and consider the extension $\varphi^\Q: N_X^\Q \to N_X^\Q$. From \cite{dm05-1} it follows that some power $(\varphi^\Q)^k$ has the same characteristic polynomial as an automorphism $\psi$ induced by automorphisms $\psi_i$ on the $N_{\lambda_i}^\Q$ with $\lambda_i$ the coherent components. As $(\varphi^\Q)^k$ and hence also $\psi$ is integer-like, we thus get that the automorphisms $\psi_i$ also are integer-like. By Lemma \ref{lem:heis}, the groups $N_{\lambda_i}^\Q$ are not isomorphic to $\Q$ or $H_3(\Q)$. In particular, for every coherent component $\lambda_i$ we have that either $\lambda_i$ consists of three elements, or $N_{\lambda_i}^\Q$ is isomorphic to $\Q^2$ and thus $\lambda_i$ is the empty graph.
\end{proof}

\noindent Again, the condition of Theorem \ref{thm:graph} is equivalent to the condition for having an Anosov automorphism, see \cite{dm05-1}.

At this point it is worth to point out that the \TG s having an Anosov automorphism are rather rare. Indeed, it is shown in \cite[Theorem 5.1.]{dw23-1} that the nilpotent group associated to a random graph does not support an Anosov automorphism. This shows that it is to be expected that only few groups satisfy Question \ref{QN1}, or thus that in order to check Conjecture \ref{con} we have a strong restriction on the possible examples.

\subsection{Open questions}

In all examples of this paper, we used Corollary \ref{cor:anotame} to show that a \TG~$N$ admits a tame automorphism. It is unclear whether there exists a \TG~which only admits tame automorphisms of a different type.

\begin{QN}
	Does there exist a \TG~$N$ with a tame automorphism $\varphi: N \to N$ such that $N$ does not admit an Anosov automorphism?
\end{QN}

\noindent Of course, there exist tame automorphisms that are not Anosov, of which we present one below. However, in all known examples, the group has other automorphisms that are Anosov.

\begin{Ex}
	Consider the group $\Z^4$, where automorphisms are given by integer matrices having determinant $\pm1$. The matrix $$A = \begin{pmatrix}
		0 & 0 & 0 & -1 \\
		1 & 0 & 0 & 2\\
		0 & 1 & 0 & 0\\
		0 & 0 & 1 & 2\\
	\end{pmatrix},
	$$ has characteristic polynomial equal to $p(x)=x^4-2x^3-2x +1$. Next to two real roots different from $\pm1$, $p(x)$ also has roots $$\frac{1 -\sqrt{3} \pm  i  \sqrt[4]{12}}{2},$$ and a direct computation shows that these have absolute value $1$, but are not roots of unity. So $A$ induces a tame automorphism that is not Anosov. The group $\Z^4$ also has Anosov automorphisms, for example the one induced by the matrix $$B = \begin{pmatrix}
		2 & 1 & 0 & 0\\
		1 & 1 & 0 & 0\\
		0 & 0 & 2 & 1\\
		0 & 0 & 1 & 1
	\end{pmatrix}.$$
\end{Ex}

On the other hand, to show that an example does not admit a tame automorphism, there always existed a fixed power $n>0$ such that every $\varphi^n$ has an infinite Reidemeister number. For example, on $\Z$ and $H_3(\Z)$ we had that $R(\varphi^2) = \infty$ for all automorphisms $\varphi$, see Example \ref{ex:heis}. We wonder whether there are groups for which this property does not hold.

\begin{QN}
Find a group $G$ without a tame automorphism, but such that there does not exist some fixed $n > 0$ for $G$ with $R(\varphi^n) = \infty$ for all $\varphi \in \Aut(G)$.
\end{QN}
\noindent It is unclear if such examples can be found in the class of \TG s. If such $n$ exists for a group $G$, one can define $n(G)$ as the minimal $n > 0$ such that $R\left(\varphi^{n} \right) = \infty$ for all $\varphi \in \Aut(G)$. In particular, a group $G$ has $R_\infty$ if and only if $n(G) = 1$, and we have that $n(\Z) = n(H_3(\Z)) = 2$.

Finally, we focused on the case of automorphisms on \TG s in this paper. However, the question remains what can be said for general monomorphisms $\varphi: \Gamma \to \Gamma$ on arbitrary virtually nilpotent groups. 
\begin{QN}
Given a virtually nilpotent group $\Gamma$ with a tame monomorphism $\varphi: \Gamma \to \Gamma$, what can we say about the structure of $\Gamma$? 
\end{QN}
\noindent An answer to this question will help to solve Conjecture \ref{con}. There is a special class of monomorphisms that is always tame, namely the expanding ones. However, there exist  \TG s having no expanding monomorphism, but a tame automorphism (indeed, even an Anosov automorphism), see \cite{dere14-1,dw22-1}. Hence the properties of this class remain unclear at the moment.

\bibliography{ref}
\bibliographystyle{plain}

\end{document}